\documentclass[a4paper, 11pt]{amsproc}

\usepackage{amsmath, amsthm, amssymb, mathtools, mathrsfs, stmaryrd}
\usepackage{ascmac}
\usepackage{comment}
\usepackage{bm}

\usepackage{hyperref}

\usepackage{graphicx}
\usepackage{here}
\usepackage{time}
\usepackage[abbrev]{amsrefs}

\usepackage{xcolor}
\usepackage[capitalize,nameinlink,noabbrev,nosort]{cleveref}
\hypersetup{
	colorlinks=true,       
	linkcolor=blue,          
	citecolor=blue,        
	filecolor=blue,      
	urlcolor=blue,           
}

\usepackage{fullpage}

\makeatletter
\@namedef{subjclassname@2020}{%
  \textup{2020} Mathematics Subject Classification}
\makeatother

\newtheorem{theoremcounter}{Theorem Counter}[section]

\theoremstyle{definition}
\newtheorem{definition}[theoremcounter]{Definition}

\newtheorem{example}[theoremcounter]{Example}

\theoremstyle{plain}
\newtheorem{lemma}[theoremcounter]{Lemma}
\newtheorem{proposition}[theoremcounter]{Proposition}

\newtheorem{theorem}[theoremcounter]{Theorem}

\numberwithin{equation}{section}

\newcommand{\Z}{\mathbb{Z}}

\begin{document}

\title[]{How long can $k$-G\"{o}bel sequences remain integers?} 

\author{Rinnosuke Matsuhira}
\address{Department of Mathematics, Kyushu University,
Motooka 744, Nishi-ku, Fukuoka 819-0395, Japan}
\email{matsuhira.rinnosuke.245@s.kyushu-u.ac.jp} 

\author{Toshiki Matsusaka}
\address{Faculty of Mathematics, Kyushu University,
Motooka 744, Nishi-ku, Fukuoka 819-0395, Japan}
\email{matsusaka@math.kyushu-u.ac.jp} 

\author{Koki Tsuchida}
\address{Department of Mathematics, Kyushu University,
Motooka 744, Nishi-ku, Fukuoka 819-0395, Japan}
\email{tsuchida.koki.481@s.kyushu-u.ac.jp} 

\subjclass[2020]{Primary 11B37; Secondary 11B50}



\maketitle

\begin{abstract}
	Inspired by Episode 3 of the Japanese manga ``Seisu-tan" by Doom Kobayashi and Shin-ichiro Seki, we investigate the $k$-G\"{o}bel sequence $(g_{k,n})_n$ named after Fritz G\"{o}bel. Although the sequence is generally defined as rational, quite a few initial terms behave like an integer sequence. This article addresses a question raised in Seisu-tan and shows that $g_{k,n}$ is always an integer for any $k \geq 2$ and $0 \leq n \leq 18$.
\end{abstract}


\section{Introduction}

As Richard Guy~\cite{Guy1988} pointed out, it can be challenging to guess the underlying pattern of a sequence from just a few examples. Fritz G\"{o}bel\footnote{Although Guy's writings do not provide any information about G\"{o}bel's identity, Neil Sloane has informed the authors that G\"{o}bel (Enschede, The Netherlands) is the author of the article~\cite{Gobel1970}.} investigated a remarkable sequence $(g_n)_n$ defined by the recursion
\[
	g_n = \frac{1 + g_0^2 + g_1^2 + \dots + g_{n-1}^2}{n}
\]
with the initial value $g_0 = 1$, (see~\cite[E15]{Guy2004}). 
Although the sequence $(g_n)_n$ starting as 1, 2, 3, 5, 10, 28, 154, 3520, 1551880, 267593772160, $\dots$, seems to be an integer sequence, Hendrik Lenstra~\cite{Guy1981} made the discovery that $g_{43} \approx 5.4 \times 10^{178485291567}$ is not an integer. 
The history can be found in the letter from Lenstra to Neil Sloane dated May 13, 1975, available in~\cite[A003504]{OEIS}. Here is an excerpt from the letter.

\begin{quote}
Dear Dr. Sloane,

\noindent Thank you very much for sending me the reprint + the first supplement.

\noindent The sequence from my letter of April 14:
\begin{align*}
	a_1 &= 1\\
	a_{n+1} &= \frac{1+a_1^2+a_2^2 + \cdots + a_n^2}{n}
\end{align*}
was mentioned to me by F. G\"{o}bel, when he saw my copy of your book. I was able to explain its absence by proving
\[
	a_n \in \mathbb{Z} \Longleftrightarrow n \leq 43.
\]
$\dots$

\noindent With kindest regards,

\noindent H. W. Lenstra, Jr.
\end{quote}

Additionally, David Boyd, Alfred van der Poorten~\cite[E15]{Guy1981}, and Henry Ibstedt~\cite{Ibstedt1990} examined the sequence obtained by replacing the squares with cubes in the above definition of $g_n$. They showed that it remains an integer sequence until the 88th term, but its integrality property breaks at the 89th. This observation led to a more general concept of $k$-G\"{o}bel sequences.

\begin{definition}
	For an integer $k \geq 2$, the \emph{$k$-G\"{o}bel sequence} $(g_{k,n})_n$ is defined by $g_{k,0} = 1$ and
	\[
		g_{k,n} = \frac{1}{n} \left(1 + \sum_{j=0}^{n-1} g_{k,j}^k \right)
	\]
	for $n \geq 1$.
\end{definition}

Following Lenstra and Ibstedt's considerations, we investigate how long $k$-G\"{o}bel sequences remain integers. For each integer $k \geq 2$, we put $N_k = \inf \{n \in \mathbb{Z}_{\geq 0} : g_{k,n} \not\in \mathbb{Z}\}$. If $g_{k,n}$ is always an integer for any non-negative integer $n$, then $N_k = \infty$. The first several terms are calculated as shown in the table below.
\begin{table}[H]
\centering
\begin{tabular}{c|cccccccccc}
	\hline
	$k$ & 2 & 3 & 4 & 5 & 6 & 7 & 8 & 9 & 10 & 11  \\
	$N_k$ & 43 & 89 & 97 & 214 & 19 & 239 & 37 & 79 & 83 & 239\\
	\hline
	$k$  & 12 & 13 & 14 & 15 & 16 & 17 & 18 & 19 & 20 & 21 \\
	$N_k$ & 31 & 431 & 19 & 79 & 23 & 827 & 43 & 173 & 31 & 103\\
	\hline
	$k$  & 22 & 23 & 24 & 25 & 26 & 27 & 28 & 29 & 30 & 31 \\
	$N_k$ & 94 & 73 & 19 & 243 & 141 & 101 & 53 & 811 & 47 & 1077\\	
	\hline
	$k$  & 32 & 33 & 34 & 35 & 36 & 37 & 38 & 39 & 40 & 41 \\
	$N_k$ & 19 & 251 & 29 & 311 & 134 & 71 & 23 & 86 & 43 & 47\\	
	\hline
	$k$  & 42 & 43 & 44 & 45 & 46 & 47 & 48 & 49 & 50 & 51 \\
	$N_k$ & 19 & 419 & 31 & 191 & 83 & 337 & 59 & 1559 & 19 & 127\\	
	\hline
	$k$  & 52 & 53 & 54 & 55 & 56 & 57 & 58 & 59 & 60 & 61 \\
	$N_k$ & 109 & 163 & 67 & 353 & 83 & 191 & 83 & 107 & 19 & 503\\	
	\hline
\end{tabular}
\caption{OEIS~\cite[A108394]{OEIS}}
\label{Nk-list}
\end{table}

In Stewart's book~\cite[Life, Recursion and Everything]{Stewart2010}, the first few terms of $N_k$ were introduced and described as follows: ``As far as I know, no one really understands why these sequences behave like they do". More recently, in the Japanese manga ``Seisu-tan"\footnote{The title ``Seisu-tan" has a double meaning in Japanese. One is ``the tale (tan) of integers (seisu)". The other involves using the term ``tan" (commonly employed as a suffix in anime and manga character names) to anthropomorphize integers.}~\cite[Episode 3]{KobayashiSeki2023}, the $k$-G\"{o}bel sequence is addressed, and the following question is posed: ``What is the minimum value of $N_k$?". In this article, we answer the question.

\begin{figure}[H]
 \centering
 \includegraphics[keepaspectratio,scale=0.25]{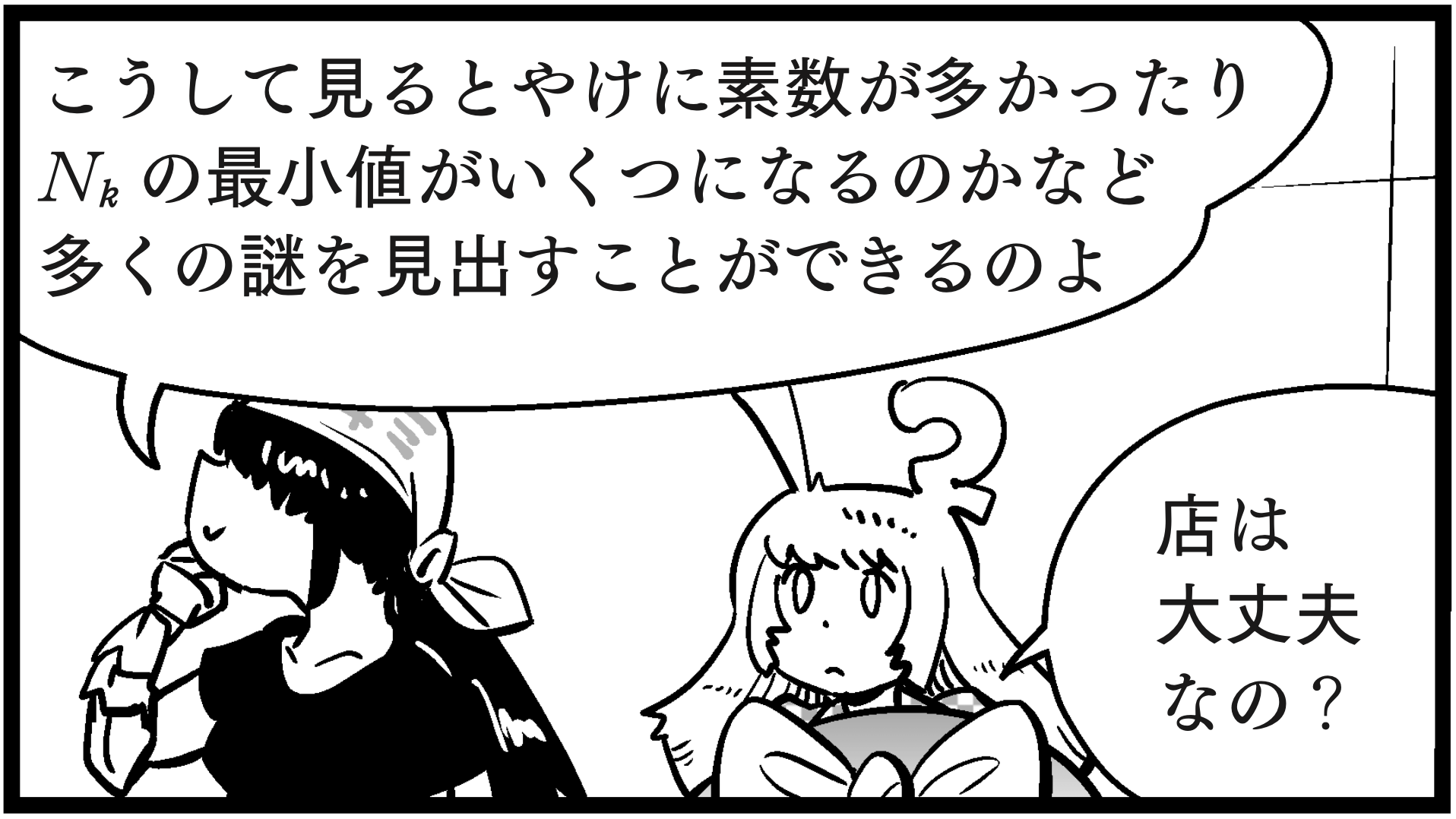}
 \caption{The excerpt from Seisu-tan~\cite[Episode 3]{KobayashiSeki2023}. They observed the list of $N_k$ and asked the above question.}
 \label{Figure}
\end{figure}

\begin{theorem}\label{main}
	We have $\min_{k \geq 2} N_k = 19$, which implies that $g_{k,n} \in \mathbb{Z}$ for any $k \geq 2$ and $0 \leq n \leq 18$. Moreover, we have $N_k = 19$ if and only if $k \equiv 6,14 \pmod{18}$.
\end{theorem}

\section{Proof}

\subsection{Settings}

The rough idea of the first half of our proof is in line with that for the $2$-G\"{o}bel sequence explained in~\cite{KobayashiSeki2023}. To prove \cref{main}, we first transform the recursive formula defining the $k$-G\"{o}bel sequence into an alternative expression for simpler calculations.

\begin{lemma}\label{rec-k-goebel}
	For any $n \geq 2$, we have
	\[
		n g_{k,n} = g_{k,n-1} (n-1 + g_{k,n-1}^{k-1})
	\]
	with $g_{k,1} = 2$.
\end{lemma}

\begin{proof}
	By definition, we see that
	\begin{align*}
		n g_{k,n} = 1 + \sum_{j=0}^{n-1} g_{k,j}^k = (n-1) g_{k,n-1} + g_{k,n-1}^k,
	\end{align*}
	which implies the result.
\end{proof}

To prove the integrality property $g_{k,n} \in \mathbb{Z}$ for any $k \geq 2$ and $1 \leq n \leq 18$, it is enough to show that $g_{k,n} \in \mathbb{Z}_{(p)}$ for any prime number $p$ since we have
\[
	\bigcap_{p \in \mathcal{P}} \mathbb{Z}_{(p)} = \mathbb{Z},
\]
where $\mathcal{P}$ is the set of all prime numbers, and $\mathbb{Z}_{(p)}$ is the localization of $\mathbb{Z}$ at $(p)$ defined by $\mathbb{Z}_{(p)} = \left\{a/b \in \mathbb{Q} : p \nmid b\right\}$. For $p \geq 19$, by \cref{rec-k-goebel}, we obviously see that $g_{k,n} \in \mathbb{Z}_{(p)}$ for any $k \geq 2$ and $1 \leq n \leq 18$. As for the remaining primes $2 \leq p \leq 17$, we need some calculations. We begin by explaining our strategy by using examples.

\begin{example}
When $p = 17$ and $k=2$, by repeatedly applying \cref{rec-k-goebel}, we can verify that $g_{2,n} \in \mathbb{Z}_{(17)}$ for all $1 \leq n \leq 16$ and $g_{2,16} \equiv 1 \pmod{17}$ in $\mathbb{Z}_{(17)}$. Thus we have $17 g_{2,17} = g_{2,16}(16 + g_{2,16}) \equiv 0 \pmod{17}$, which implies $g_{2,17} \in \mathbb{Z}_{(17)}$. We also have $g_{2,18} \in \mathbb{Z}_{(17)}$. 
\end{example}

\begin{example}\label{example-p7}
When $p=7$ and $k=2$, the first step is to confirm that $g_{2,6} \equiv 7 \pmod{7^2}$ by considering congruences modulo higher powers (corresponding to the largest power of $p$ that divides $18!$). Then, the congruence $7g_{2,7} = g_{2,6}(6+g_{2,6}) \equiv 7 \cdot 13 \pmod{7^2}$ implies that $g_{2,7} \equiv 6 \pmod{7}$ in $\mathbb{Z}_{(7)}$. Subsequently, we have $g_{2,13} \equiv 1 \pmod{7}$ and $14 g_{2,14} = g_{2,13}(13+g_{2,13}) \equiv 0 \pmod{7}$, which implies $g_{2,14} \in \mathbb{Z}_{(7)}$. Here is the list of all the calculations we have done.
\begin{align*}
	g_{2,1} &\equiv 2 \pmod{7^2}, & g_{2,2} &\equiv 3 \pmod{7^2}, & g_{2,3} &\equiv 5 \pmod{7^2}, & g_{2,4} &\equiv 10 \pmod{7^2},\\
	g_{2,5} &\equiv 28 \pmod{7^2}, & g_{2,6} &\equiv 7 \pmod{7^2}, & g_{2,7} &\equiv 6 \pmod{7^1}, & g_{2,8} &\equiv 1 \pmod{7^1},\\
	g_{2,9} &\equiv 1 \pmod{7^1}, & g_{2,10} &\equiv 1 \pmod{7^1}, & g_{2,11} &\equiv 1 \pmod{7^1}, & g_{2,12} &\equiv 1 \pmod{7^1},\\
	g_{2,13} &\equiv 1 \pmod{7^1}, & g_{2,14} &\equiv 0 \pmod{7^0}, & g_{2,15} &\equiv 0 \pmod{7^0}, & g_{2,16} &\equiv 0 \pmod{7^0},\\
	g_{2,17} &\equiv 0 \pmod{7^0}, & g_{2,18} &\equiv 0 \pmod{7^0}.
\end{align*}
In particular, we have $g_{2,n} \in \mathbb{Z}_{(7)}$ for all $1 \leq n \leq 18$. We remark that continuing this calculation does not show $g_{2,21} \in \mathbb{Z}_{(7)}$. To show it, we have to consider congruences modulo $7^3$.
\end{example}

\begin{example}\label{example-k6p19}
	When $p=19$ and $k=6$, by the same argument, we have $g_{6,n} \in \Z_{(19)}$ for all $1 \leq n \leq 18$ and $g_{6,18} \equiv 16 \pmod{19}$. Thus we have $19 g_{6,19} = g_{6,18}(18 + g_{6,18}^5) \equiv 10 \pmod{19}$, which implies $g_{6,19} \notin \Z_{(19)}$.
\end{example}

Let $\nu_p(n)$ be the exponent of the largest power of $p$ that divides $n$. To compute the general cases by using Mathematica, we rephrase \cref{rec-k-goebel} and the above examples by introducing the following sequences.

\begin{definition}
Let $k \geq 2, r \geq 1$ be integers, and $p$ a prime. For any positive integer $n$ with $\nu_p(n!) \leq r$, we define $g_{k,n,p,r} \in \mathbb{Z}/p^{r-\nu_p(n!)}\mathbb{Z} \cup \{\mathsf{F}\}$ by the initial value $g_{k,1,p,r} = 2 \bmod{p^r} \in \mathbb{Z}/p^r \mathbb{Z}$ and the following recursion: When $g_{k,n-1,p,r} = a \bmod{p^b}$ with $b = r-\nu_p((n-1)!)$, we define
\begin{align*}
	g_{k,n,p,r} = \begin{cases}
		\dfrac{a(n-1+a^{k-1})}{p^{\nu_p(n)}} (n/p^{\nu_p(n)})^{-1} \bmod{p^{b - \nu_p(n)}} &\text{if } b-\nu_p(n) > 0,\\
		0 \bmod{p^0} &\text{if } b-\nu_p(n) = 0
	\end{cases}
\end{align*}
if $a(n-1 + a^{k-1}) \equiv 0 \pmod{p^{\nu_p(n)}}$ and $g_{k,n,p,r} = \mathsf{F}$ if otherwise, where $(n/p^{\nu_p(n)})^{-1}$ is the inverse in $(\mathbb{Z}/p^{b-{\nu_p(n)}}\mathbb{Z})^\times$. When $g_{k,n-1,p,r} = \mathsf{F}$, we define $g_{k,n,p,r} = \mathsf{F}$.
\end{definition}

We easily see that if $g_{k,n,p,r} \neq \mathsf{F}$, then $g_{k,n} \in \Z_{(p)}$ and $g_{k,n,p,r} = g_{k,n} \bmod{p^{r-\nu_p(n!)}}$ in $\mathbb{Z}_{(p)}/p^{r-\nu_p(n!)}\mathbb{Z}_{(p)} \cong \mathbb{Z}/p^{r-\nu_p(n!)}\mathbb{Z}$. Moreover, if $g_{k,n-1,p,r} \neq \mathsf{F}$ and $g_{k,n,p,r} = \mathsf{F}$, then $g_{k,n} \notin \mathbb{Z}_{(p)}$ holds. The symbol ``$\mathsf{F}$" is an abbreviation for ``false". Since $g_{k,n,p,r} \neq \mathsf{F}$ implies that $g_{k,m,p,r} \neq \mathsf{F}$ for all $1 \leq m \leq n$, all we have to do is to check that $g_{k,18,p,\nu_p(18!)} \neq \mathsf{F}$ for all $k \geq 2$ and primes $2 \leq p \leq 17$. The above examples will be explained again in terms of the $g_{k,n,p,r}$ in \cref{gknpr-ex} below.

\subsection{Reduction to the finite number of $k$'s}

By focusing on the periodicity of $g_{k,n,p,r}$ for $k$, we show that it is sufficient to check the above calculations for finite cases.

\begin{lemma}\label{Euler-theorem}
	Let $r \geq 1$ be a positive integer and $p$ a prime. We assume that integers $k,\ell$ satisfy $r+1 \leq k \leq \ell$ and $k \equiv \ell \pmod{\varphi(p^r)}$, where $\varphi(n)$ is the Euler totient function. Then, for any integers $a, b \in \mathbb{Z}$ with $0 \leq b \leq r$, we have
	\[
		a^{k-1} \equiv a^{\ell-1} \pmod{p^b}.
	\]
\end{lemma}

\begin{proof}
	By Euler's theorem, we have
	\[
		a^{\ell-1} - a^{k-1} = a^{k-1}(a^{\ell-k}-1) \equiv 0 \pmod{p^r}
	\]
	whether $p \nmid a$ or $p \mid a$.
\end{proof}

\begin{proposition}\label{periodic}
	Under the same assumptions as in \cref{Euler-theorem}, for any integer $n$ satisfying $\nu_p(n!) \leq r$, we have $g_{k,n,p,r} = g_{\ell, n, p, r}$.
\end{proposition}

\begin{proof}
	It follows from the induction on $n$ and \cref{Euler-theorem}.
\end{proof}

Therefore, we have to check that $g_{k,18,p,r} \neq \mathsf{F}$ for primes $2 \leq p \leq 17$ and $2 \leq k \leq \varphi(p^r)+r$ with $r = \nu_p(18!)$. Finally, we will check them by using Mathematica.

\subsection{Mathematica}

The following is a code for computing $k$-G\"{o}bel sequence $g_{k,n,p,r}$.

\begin{screen}
\begin{verbatim}
nu[p_, n_] := FirstCase[FactorInteger[n], {p, r_} -> r, 0];
inv[n_, M_] := If[M == 1, 1, ModularInverse[n, M]];

g[k_, 1, p_, r_] := {2, r};
g[k_, n_, p_, r_] := 
  g[k, n, p, r] = 
   Module[{a, b}, 
    If[g[k, n - 1, p, r] === "F", "F", {a, b} = g[k, n - 1, p, r];
     If[Mod[a (n - 1 + a^(k - 1)), p^nu[p, n]] != 0, 
      "F", {Mod[
        a (n - 1 + a^(k - 1))/p^nu[p, n] inv[n/p^nu[p, n], 
          p^(b - nu[p, n])], p^(b - nu[p, n])], b - nu[p, n]}]]];
\end{verbatim}
\end{screen}

\begin{example}\label{gknpr-ex}
	We show the lists of $(g_{2,n,7,2})_{n=1}^{18}$ and $(g_{6,n,19,1})_{n=1}^{19}$.
	\begin{table}[H]
	\centering
	\begin{tabular}{c|cccccccccccccccccc}
		\hline
		$n$ & 1 & 2 & 3 & 4 & 5 & 6 & 7 & 8 & 9 & 10 \\
		$g_{2,n,7,2} $ & \{2,2\} & \{3,2\} & \{5,2\} & \{10,2\} & \{28,2\} & \{7,2\} & \{6,1\} & \{1,1\} & \{1,1\} & \{1,1\} \\
		\hline
		$n$ & 11 & 12 & 13 & 14 & 15 & 16 & 17 & 18 \\
		$g_{2,n,7,2} $ & \{1,1\} & \{1,1\} & \{1,1\} & \{0,0\} & \{0,0\} & \{0,0\} & \{0,0\} & \{0,0\} \\
		\hline
	\end{tabular}
	\end{table}
	\begin{table}[H]
	\centering
	\begin{tabular}{c|cccccccccccccccccc}
		\hline
		$n$ & 1 & 2 & 3 & 4 & 5 & 6 & 7 & 8 & 9 & 10 \\
		$g_{6,n,19,1} $ & \{2,1\} & \{14,1\} & \{18,1\} & \{9,1\} & \{17,1\} & \{9,1\} & \{12,1\} & \{13,1\} & \{17,1\} & \{16,1\} \\
		\hline
		$n$ & 11 & 12 & 13 & 14 & 15 & 16 & 17 & 18 & 19\\
		$g_{6,n,19,1} $ & \{10,1\} & \{18,1\} & \{5,1\} & \{16,1\} & \{4,1\} & \{8,1\} & \{2,1\} & \{16,1\} & $\mathsf{F}$\\
		\hline
	\end{tabular}
	\end{table}
	The output $\{a, b\}$ for $n$ means that $g_{k,n,p,r} = a \bmod{p^b}$, that is, $g_{k,n} \equiv a \pmod{p^b}$ in $\mathbb{Z}_{(p)}$. The results coincide with those in \cref{example-p7} and \cref{example-k6p19}.
\end{example}



The resulting outputs of $g_{k,18,p,\nu_p(18!)}$ for all cases we have to check are $\{0,0\}$. In other words, as desired, we have successfully confirmed that $g_{k,18,p,\nu_p(18!)} \neq \mathsf{F}$ for primes $2 \leq p \leq 17$ and integers $2 \leq k \leq \varphi(p^{\nu_p(18!)}) + \nu_p(18!)$. This concludes the proof of the first half of \cref{main}.

	
	
	
	
	
	

Furthermore, we also have the following table for $p=19$, which immediately implies the second half of \cref{main}, that is, $N_k = 19$ if and only if $k \equiv 6, 14 \pmod{18}$.
\begin{table}[H]
\centering
\begin{tabular}{c|cccccccccccccccccc}
	\hline
	$k$ & 2 & 3 & 4 & 5 & 6 & 7 & 8 & 9 & 10 \\
	$g_{k,19,19,1} $ & \{0,0\} & \{0,0\} & \{0,0\} & \{0,0\} & $\mathsf{F}$ & \{0,0\} & \{0,0\} & \{0,0\} & \{0,0\} \\
	\hline
	$k$ & 11 & 12 & 13 & 14 & 15 & 16 & 17 & 18 & 19\\
	$g_{k,19,19,1} $ & \{0,0\} & \{0,0\} & \{0,0\} & $\mathsf{F}$ & \{0,0\} & \{0,0\} & \{0,0\} & \{0,0\} & \{0,0\}\\
	\hline
\end{tabular}
\end{table}

\section{Concluding remarks}

Let $N$ be the set of all $N_k$ given by
\begin{align*}
	N = \{N_k \in \mathbb{Z}_{>0} \cup \{\infty\} : k \geq 2\}.
\end{align*}
Determining the set $N$ is a fundamental question requiring further investigation. Our \cref{main} ensures that the minimum of $N$ is $19$, and \cref{Nk-list} suggests that
\[
	N = \{19, 23, 29, 31, 37, 43, \dots\}.
\]
Upon observing its initial terms, only prime numbers may seem to occur. However, for example, $N_5 = 214$ is a composite number, and by using the same approach as in our proof, it can be confirmed that the prime number $41$ is not included in $N$. As discussed in~\cite{KobayashiSeki2023}, it is not known whether $N_k$ always takes a finite value for any $k \geq 2$, and whether $\sup N$ is infinite. 

Zagier~\cite[Day 5, Problem 3]{Zagier1996} addressed the ($2$-)G\"{o}bel sequence and provided an asymptotic formula for $g_n$, (see also~\cite[6.10]{Finch2003} and \cite{Weisstein}). He also discussed heuristics suggesting that 
$g_p \in \mathbb{Z}_{(p)}$ often holds. 

We introduce another approach by Ibstedt~\cite{Ibstedt1990}. Specifically, a new sequence can be obtained by adopting the recursion in \cref{rec-k-goebel} as the definition of $k$-G\"{o}bel sequences and varying the initial value choice of $g_{k,1} = 2$. Consequently, $N_k$ can be defined similarly. For instance, the $2$-G\"{o}bel sequence with the initial value $g'_{2,1} = 3$ is given by
\[
	(g'_{2,n})_n = \left(3, 6, 16, 76, 1216, 247456, \frac{61235956672}{7}, \dots \right).
\]
This example illustrates that altering the initial value can cause the loss of integrality property earlier than in the original $k$-G\"{o}bel sequence. In these cases as well, similar problems are likely to be considered.

\section*{Acknowledgements}

The authors would like to express their gratitude to Neil Sloane for providing background information on Fritz G\"{o}bel and to Masanobu Kaneko for sharing a copy of Zagier's note~\cite{Zagier1996} and offering helpful comments. Additionally, they extend their thanks to Shin-ichiro Seki for inspiring this study and providing valuable comments. Special acknowledgements are also due to Doom Kobayashi and Akira Iino for granting permission to use the image in \cref{Figure}. The second author was supported by JSPS KAKENHI Grant Numbers JP20K14292 and JP21K18141.

\bibliographystyle{amsplain}
\bibliography{References}

\end{document}